\documentclass[12pt,reqno]{amsart}
\usepackage{amssymb, amsmath,amsthm}
\usepackage{graphicx}
\usepackage{subfig}
\newtheorem{thm}{Theorem}

\newtheorem{defn}{Definition}

\numberwithin{defn}{section}
\numberwithin{thm}{section}
\numberwithin{Lemma}{section}
\numberwithin{Corollary}{section}
\numberwithin{Example}{section}
\numberwithin{subsection}{section}
\numberwithin{Remark}{section}
\numberwithin{equation}{section}
\numberwithin{ppn}{section}
\begin{document}
\title[ A class of optimal eighth-order Steffensen-type ... ]
{ A class of optimal eighth-order Steffensen-type  iterative methods for solving nonlinear equations and their basins of attraction} 
\author{Anuradha Singh and J. P. Jaiswal }
\date{}
\maketitle


\textbf{Abstract.} 
This article concerned with the issue of solving a nonlinear equation with the help of iterative method where no any derivative evaluation is required per iteration. Therefore, this work contributes to a new class of optimal eighth-order Steffensen-type methods. Theoretical proof has been given to reveal the eighth-order convergence. Numerical comparisons have been carried out to show the effectiveness of contributed scheme.
\section{Introduction}
During the recent past,  a wide collection of iterative methods has been presented  in many journals, one can  see [\cite{Soleymani1}-\cite{Soleymani11}] and the references therein. In order to find the solution of a nonlinear equation Newton has provided the following iterative formula
\begin{equation}
x_{n+1}=x_n- \frac{f(x_n)}{f'(x_n)}.
\end{equation}
Steffensen \cite{steffensen} was the first who furnished the derivative-free form of Newton's scheme given by :
\begin{equation}
x_{n+1}=x_n-\frac{f(x_n)}{f[x_n,w_n]}, \  \   \   \  w_n=x_n+f(x_n), \ \  n=0,1,2...
\end{equation}
both schemes possess the quadratic rate of convergence and same efficiency index 1.414. Although both the methods have the same order of convergence and efficiency index,  but Steffensen method is derivative free.  In order to increase the rate of convergence and efficiency index of iterative methods the number of function evaluations may increase. Kung and Traub \cite{Kung} conjectured that a multipoint iteration without memory consuming $n$ evaluation per full iteration can reach the maximum convergence rate $2^{n-1}$. A large collection of research papers is available on the higher-order iterative methods agree with the Kung-Traub conjecture. In order to compare different iterative methods of same order the classical efficiency index of an iterative process in \cite{Traub1} given by $p^{\frac{1}{n}}$, where $p$ is the rate of convergence and $n$ is the total number of functional evaluations per iteration. More recently, many researchers have focused to make existing iterative methods free from derivatives, interested researcher can follow \cite{Soleymani3}-\cite{Soleymani6} . In many of the science and engineering problem, the evaluation of derivative is difficult and time consuming. Therefore, the Steffensen-type methods have become very popular in terms of solving nonlinear equations.
This study is summarized as follows: Firstly, we provide a brief review of available literature to reveal the development of different  derivative-free iterative methods. In the next section, we design a new optimal eight-order Steffensen-type iterative method for finding simple roots of nonlinear equations .  In section 4, we employ some numerical examples to compare the performance of the new method  with some existing eight-order derivative-free methods. Section 5, reveals the  graphical comparison by basins of attraction. Finally, in the last section  brief conclusion  will be given.

\section{A brief review of the available literature}

In this section,  we give an overview of some recent derivative free methods.
 Soleymani et al. \cite{Soleymani7}  have improved the efficiency index of  following method in terms of making it derivative-free 
\begin{eqnarray*}
y_n&=&x_n- \frac{f(x_n)}{f[x_n,w_n]} , \  \  w_n=x_n+ \beta f(x_n), \   \beta \in \mathbb{R} \setminus \{ 0\}, \nonumber\\
z_n&=&y_n - \frac{f(y_n)f(w_n)}{(f(w_n)-f(y_n))f[x_n,y_n]}, \nonumber\\
x_{n+1}&=& z_n -\frac{f(z_n)}{f'(z_n)}.
\end{eqnarray*}
This method has the eighth-order convergence and $1.516$ as its efficiency index. To improve its efficiency index authors have established  two optimal three-step multipoint derivative-free methods given by  
\begin{eqnarray*}
y_n&=&x_n -\frac{f(x_n)}{f[x_n,w_n]},     w_n=x_n+ \beta f(x_n), \nonumber\\
 z_n&=&y_n-\frac{f(y_n)f(w_n)}{(f(w_n)-f(y_n))f[x_n,y_n]}, \nonumber\\
x_{n+1}&=&z_n- \frac{f(z_n)f(w_n)}{(f(w_n)-f(y_n))f[x_n,y_n]} \{G(\varphi) \times H(\tau) \times Q(\sigma) \times L(\rho)\}, 
\end{eqnarray*}
and 
\begin{eqnarray*}
y_n&=&x_n -\frac{f(x_n)}{f[x_n,w_n]},     w_n=x_n- \beta f(x_n), \nonumber\\
 z_n&=&y_n-\frac{f(y_n)f(w_n)}{(f(w_n)-f(y_n))f[x_n,y_n]}, \nonumber\\
x_{n+1}&=&z_n- \frac{f(z_n)f(w_n)}{(f(w_n)-f(y_n))f[x_n,y_n]} \{G(\varphi) \times H(\tau) \times Q(\sigma) \times L(\rho)\}, 
\end{eqnarray*}
where $\beta \in \mathbb{R} \setminus \{0\}$, $\varphi= \frac{f(z)}{f(y)}$ , $\tau=\frac{f(z)}{f(w)}$, $\sigma=\frac{f(z)}{f(x)}$, $\rho=\frac{f(y)}{f(w)}$. These methods have eighth-order convergence with efficiency index $1.682$ under some conditions on the weight functions given in the same paper.\\
In \cite{Soleymani2}, Soleymani has accelerated the efficiency index of the following eighth-order multipoint structure
 \begin{equation*}
y_n=x_n- \frac{f(x_n)}{f[x_n,w_n]}, \ \  z_n=y_n- \frac{f(y_n)}{f'(y_n)}, \ \  x_{n+1}=z_n-\frac{f(z_n)}{f'(z_n)}.
\end{equation*}
He has approximated the derivatives by replacing $f'(y_n) \approx f[x_n, w_n]$, $f'(z_n) \approx f[x_n,w_n]$ and used the concept of weight functions to make it optimal as well as derivative-free. He proposed the following iterative formula
\begin{eqnarray*}
y_n&=&x_n-\frac{f(x_n)}{f[x_n,w_n]}, \  \   w_n=x_n+ \beta f(x_n), \nonumber\\
z_n&=& y_n- \frac{f(y_n)}{f[x_n,w_n]}[G(A) \times H(B)], \nonumber\\ 
x_{n+1}&=& z_n -\frac{f(z_n)}{f[x_n,w_n]} [K(\Gamma) \times L(\Delta) \times P(E) \times Q(B) \times J(A)],
\end{eqnarray*}
wherein $\beta \in \mathbb{R}\setminus \{0\} $, $A=\frac{f(y)}{f(x)}$, $B=\frac{f(y)}{f(w)}$, $\Gamma=\frac{f(z)}{f(x)}$, $\Delta=\frac{f(z)}{f(w)}$, $E=\frac{f(z)}{f(y)}$. This method has the eighth-order convergence and efficiency index $1.682$ under some conditions on weight functions.
Inspired from all these papers we also improve the order of convergence as well as the efficiency index of one existing seventh-order   method in the next section.
\section{Main method and convergence analysis}
First we give some definitions which we will use later.
\begin{defn}
Let f(x) be a real valued function with a simple root $\alpha$ and let ${x_n}$ be a sequence of real numbers that converge towards $\alpha$. The order of convergence m is given by
\begin{equation}\label{eqn:21}
\lim_{n\rightarrow\infty}\frac{x_{n+1}-\alpha}{(x_n-\alpha)^m}=\zeta\neq0,  
\end{equation}     
\noindent
where $\zeta$ is the asymptotic error constant and $m \in R^+$.\\
\end{defn}

\begin{defn}
Let $n$ be the number of function evaluations of the new method. The efficiency of the new method is measured by the concept  of efficiency index \cite{Gautschi,Traub1} and defined as
\begin{equation}\label{eqn:23}
m^{1/n},
\end{equation}
where $m$ is the order of convergence of the new method.\\
\end{defn}
Consider the following seventh-order method established by Soleymani et al. \cite{Soleymani1} to build a new eighth-order  method:
\begin{eqnarray}\label{eqn:11}
y_n&=&x_n-  \frac{f(x_n)}{f'(x_n)}\nonumber\\
z_n&=&y_n- \frac{f(y_n)}{f[x_n,y_n]}.{G(t_n)}\nonumber\\
x_{n+1}&=&z_n- \frac{f(z_n)}{f[y_n,z_n]}.{H(t_n)},
\end{eqnarray}
where $t_n=\frac{f(y_n)}{f(x_n)}$ and  $G(0)=G'(0)=1$, $ \left|G''(0)\right|< +\infty$; $H(0)=1, H'(0)=0,H''(0)=2$,  $ \left|H^{(3)}(0)\right|< +\infty$.\\                          
Now our aim is to develop efficient as well as  derivative-free version of the method $(\ref{eqn:11})$. For this we approx $f'(x_n)  \approx f[z_n,x_n]$, where $z_n$=$x_n+f(x_n)$ in $(\ref{eqn:11})$. Here we see that the method $(\ref{eqn:11})$ under this approximation of $f'(x_n)$ has fifth-order convergence and its error expression is given by 
\begin{eqnarray*}
e_{n+1}=\frac{ \{(1+c_1)^2 c_2^4 \} e_n^5 }{c_1^2} + O(e_n^6).
\end{eqnarray*}
Now to improve its order of convergence without using any new evaluation, we approx $f'(x_n) \approx f[z_n,x_n]$,  where $z_n=x_n+f(x_n)^2$, then its  error expression becomes 
\begin{eqnarray*}
e_{n+1}&=&\frac{c_2^2 \Bigl(2c_1^3c_2+2c_1c_3+c_2^2(-6+G''[0]) \Bigr) \Bigl(6c_1^3c_2+12c_1c_3+c_2^2(-24+H^{(3)}[0]) \Bigr) e_n^7}{12c_1^6}  \nonumber\\
&&+O(e_n^8).
\end{eqnarray*}
In fact, if we approx $f'(x_n) \approx f[z_n,x_n]$, where $z_n=x_n+ \alpha f(x_n)^n$ , $n \geq 2$, then its order of convergence remains seven.
Clearly here we use four function evaluations. So, according to Kung-Traub conjecture its maximum (optimal) possible order should be eight. 
 To do this, we consider the following iterative formula 
\begin{eqnarray}\label{eqn:21}
y_n&=&x_n- \frac{f(x_n)}{f[z_n,x_n]}\nonumber\\
w_n&=&y_n- G(t_1) .\frac{f(y_n)}{f[x_n,y_n]}\nonumber\\
x_{n+1}&=&w_n- H(t_2).\frac{f(w_n)}{f[w_n,y_n]},
\end{eqnarray}
where  $t_1=\frac{f(y_n)}{f(x_n)}$, $t_2=\frac{f[w_n,y_n]}{f[w_n,x_n]}$ and $z_n=x_n+f(x_n)^3$.                                                             The following theorem  shows that the conditions on weight functions under which proposed scheme has eighth-order convergence. 
\begin{thm}
Let us consider $\alpha \in D$ be a simple root of a sufficiently differentiable function $f: D \subseteq \mathbb{R} \rightarrow \mathbb{R}$.  If   $x_0$ is sufficiently close to the root $\alpha$. Then the method $(\ref{eqn:21})$ has eighth-order convergence, when the weight functions $G(t_1)$,   $H(t_2)$ satisfy the following conditions:
\begin{eqnarray}\label{eqn:22a}
&& G(0)=1,\  G^{'}(0)=1, \  \left|G^{(3)}(0)\right|< +\infty, \nonumber\\
&& H(1)=1,\  H^{'}(1)=0,\    H''(1)=2 \   , H^{(3)}(1)=-12, \left|H^{(4)}(1)\right|< +\infty.\nonumber\\
\end{eqnarray}
\end{thm}

\begin{proof}
With help of Taylor series and symbolic computation we find the error expression of method $(\ref{eqn:21})$. Furthermore,  by Taylor expansion around the simple root $\alpha$ in the $n^{th}$ iteration and by considering $e_n=x_n-\alpha$, $f(\alpha)=0$. We obtain
\begin{equation}\label{eqn:23a}
\begin{split}
f(x_n)=f'(\alpha)[c_1e_n+c_2e_n^2+c_3e_n^3+c_4e_n^4]+. . . +O(e_n^{10}),       
\end{split}
\end{equation}
and
\begin{equation}\label{eqn:24a}
\begin{split}
z_n= \alpha+e_n+f'(\alpha)^3\Bigl[e_n^3+3c_2e_n^4+(3c_2^2+3c_3)e_n^5 \Bigr]+. . . +O(e_n^{9})                     .  
\end{split}
\end{equation}
Subsequently, we obtain
\begin{equation}\label{eqn:25}
f(z_n)=f'(\alpha)\Bigl[e_n+c_2e_n^2+(f'(\alpha)^3+c_3)e_n^3 +(5 f'(\alpha)^3 c_2+c_4)e_n^4\Bigr]+...+O(e_n^9).
\end{equation}
 With the help of  equations (3.6)-(3.8), we obtain the Taylor's series expansion of $f[z_n,x_n]=\frac{f(z_n)-f(x_n)}{z_n-x_n}$ as follows:
\begin{eqnarray}\label{eqn:26a}
f[z_n,x_n]&=&f'(\alpha)+2f'(\alpha)c_2e_n+3f'(\alpha)c_3 e_n^2+(f'(\alpha)^4c_2+4f'(\alpha)c_4)e_n^3 \nonumber\\
&&+(3 f'(\alpha)^4(c_2^2+c_3)+5f'(\alpha)c_5)e_n^4)+. . . +O(e_n^{9}).
\end{eqnarray}
By putting the values of equations (3.6) and (3.9) in the first step of equation  $(\ref{eqn:21})$, we attain
\begin{equation}\label{eqn:27}
y_n=\alpha+c_2e_n^2-2(c_2^2-c_3)e_n^3+(f'(\alpha)^3c_2+4c_2^3-7c_2c_3+3c_4)e_n^4+. . . +O(e_n^{9}),
\end{equation}
On the other hand, we find
\begin{eqnarray}\label{eqn:27a}
f(y_n)&=&f'(\alpha)\Bigl[c_2e_n^2-2(c_2^2-c_3)e_n^3+(f'(\alpha)^3c_2+5c_2^3-7c_2c_3+3c_4)e_n^4\Bigr] \nonumber\\
&&+. . . +O(e_n^{9}).
\end{eqnarray}
Furthermore, we obtain
\begin{eqnarray}\label{eqn:2}
f[x_n,y_n]&=&f'(\alpha)+f'(\alpha)c_2e_n+f'(\alpha)(c_2^2+c_3)e_n^2+f'(\alpha)(-2c_2^3+3c_2c_3+c_4)e_n^3 \nonumber\\
&&+f'(\alpha)(4c_2^4+c_2^2(f'(\alpha)^3-8c_3)+2c_3^2+4c_2c_4+c_5)e_n^4 \nonumber\\
&&+...+O(e_n^{9}), 
\end{eqnarray}
and
\begin{eqnarray}\label{eqn:26b}
\frac{f(y_n)}{f[x_n,y_n]}&=& c_2 e_n^2+(-3c_2^2+2c_3)e_n^2+(7c_2^3+c_2(f'(\alpha)^3-10c_3)+3c_4)e_n^4 \nonumber\\
&&+. . . +O(e_n^{9}).
\end{eqnarray}
By using the equations $(\ref{eqn:26b})$, $(\ref{eqn:27a})$ and $(\ref{eqn:23a})$ in the second step of equation $(\ref{eqn:21})$, we attain
\begin{eqnarray}\label{eqn:28}
&w_n=&\alpha+(c_2-G(0)c_2)e_n^2+(-2c_2^2+3G(0)c_2^2+2c_3-2G(0)c_3-c_2^2G'(0))e_n^3\nonumber\\
&&+ . . .+O(e_n^{9}).
\end{eqnarray}
 By virtue of the above equation and considering $G(0)=1$,  $G'(0)= 1$, we acquire  
\begin{eqnarray}\label{eqn:29}
&&f(w_n)\nonumber\\
&&= \frac{1}{2} f'(\alpha) \Bigl(6c_2^3-2c_2c_3-c_2^3G''(0)\Bigr)e_n^4\nonumber\\
&&+\frac{1}{6}f'(\alpha)\Bigl(-6 f'(\alpha)^3c_2^2-108c_2^4+120c_2^2c_3-12c_3^2-12c_2c_4+27c_2^4G''(0) \nonumber\\
&&-18c_2^2c_3G''(0)-c_2^4G^{(3)}(0)\Bigr)e_n^5+ . . .+O(e_n^{9}).
\end{eqnarray}
With help of equations $(\ref{eqn:27a})$, $(\ref{eqn:28})$, $(\ref{eqn:29})$ and $(\ref{eqn:23a})$, we have
\begin{eqnarray}\label{eqn:30}
f[w_n,y_n]&=&f'(\alpha)+f'(\alpha) c_2^2 e_n^2+2 f'(\alpha)c_2(-c_2^2+c_3)e_n^3 \nonumber\\
&&+ \frac{1}{2} f'(\alpha) c_2(2c_2(f'(\alpha)^3-7c_3)+6c_4-c_2^3(-14+G''(0)))e_n^4 \nonumber\\
&&+ . . .+O(e_n^{9}),
\end{eqnarray}
and
\begin{eqnarray}\label{eqn:31}
f[w_n,x_n]&=&f'(\alpha)+f'(\alpha)c_2 e_n+f'(\alpha) c_3 e_n^2+ f'(\alpha) c_4 e_n^3 \nonumber\\
&&- \frac{1}{2} \Bigl(f'(\alpha) \{2c_2^2c_3-2c_5+c_2^4(-6+G''(0))\}\Bigr)e_n^4 \nonumber\\
&&+ . . .+O(e_n^{9}).
\end{eqnarray}
Now, putting the values of equations $(\ref{eqn:30})$, $(\ref{eqn:31})$ and $(\ref{eqn:29})$, in the last step of equation $(\ref{eqn:21})$, we find  
\begin{eqnarray*}
e_{n+1}&=&\frac{1}{2}(-1+H(1)) c_2 (2c_3+c_2^2(-6+G''(0)))e_n^4 \nonumber\\
&&+ \Bigl(2(-1+H(1))c_3^2+2(-1+H(1))c_2c_4+c_2^2\Bigl(f'(\alpha)^3(-1+H(1)) \nonumber\\
&&+c_3(20-20H(1)-H'(1)+3(-1+H(1))G''(0))\Bigr) \nonumber\\
&&+ \frac{1}{6}c_2^4(18(-6+6H(1)+H'(1))-3(-9+9H(1)+H'(1))G''(0) \nonumber\\
&&+(-1+H(1))G^{(3)}(0)\Bigr)e_n^5+ . . .+O(e_n^{9}).
\end{eqnarray*}
By putting $H(1)=1, \  H'(1)=0, \  H''(1)=2, \  H^{(3)}(1)=-12$, in the above equation the final error expression is given by 
%
%
%

\begin{eqnarray}\label{eqn:212}
e_{n+1}&=&\frac{1}{48}c_2(2c_3+c_2^2(-6+G''(0)))(24c_3^2-24c_2(f'(\alpha)c_2+4c_2^3+c_4) \nonumber\\
&&+c_2^4H^{(4)}(1))e_n^8+O(e_n^{9}). 
\end{eqnarray}

\end{proof}
\textbf{Particular Case:}\\\\
Let $ G(t_1)=\frac{1-2t_1}{1- 3t_1}$ and   $ H(t_2)=4-8t_2+7t_2^2-2t_2^3$,   
 then   the method  $(\ref{eqn:21})$ becomes 

\begin{eqnarray}\label{eqn:213}
y_n&=&x_n-  \frac{f(x_n)}{f[z_n,x_n]},  \  \  z_n=x_n+f(x_n)^3\nonumber\\
w_n&=&y_n- \Bigl(\frac{1-2t_1}{1-3t_1} \Bigr) \frac{f(y_n)}{f[x_n,y_n]}\nonumber\\
x_{n+1}&=&w_n-(4-8t_2+7t_2^2-2t_2^3).\frac{f(w_n)}{f[w_n,y_n]},
\end{eqnarray}
where  $t_1=\frac{f(y_n)}{f(x_n)}$,  $t_2=\frac{f[w_n,y_n]}{f[w_n,x_n]}$
and  its error expression becomes

\begin{eqnarray}\label{eqn:212}
e_{n+1}= -c_2c_3(-c_3^2+c_2(f'(\alpha)^3 c_2+4c_2^3+c_4)) e_n^8+O(e_n^{9}).
\end{eqnarray}
\textit{Remark.1:} By taking different values of $G(t_1)$ and $H(t_2)$ one may get a number of eight-order derivative-free iterative methods for finding the simple roots.\\
\textit{Remark.2:} 
In fact, if we put  $z_n=x_n+ \alpha. (f(x_n))^n$, $n \geq 3$ ,where $\alpha \neq 0 \in \mathbb{R}$ in the scheme $(\ref{eqn:21})$ then it gives the same order of convergence.\\
\textit{Remark.3:} In order of removing derivatives from iterative methods the number of function evaluation usually increase. But in our scheme the efficiency has increased without adding more function evaluations.

\section{Numerical results}   
The prime objective of this section is to demonstrate the performance of the new eighth-order derivative-free method.  In order to verify the effectiveness  of the proposed iterative method we have considered seven nonlinear test functions. The test non-linear functions and their roots are listed in Table-1. The entire computations reported here have been performed on the programming package $MATHEMATICA\  [8]$ using $1000$ digit floating point arithmetic using $''SetAccuraccy"$ command. 
 In Table 2 DIV. stands for divergent, NC and I stand for not convergent and indeterminate respectively. For comparing number of iterations and total number of function evaluations, we have used the following stopping criterion    
$|f(x_{n+1})|< 1.E-50.$
We have taken three different initial guesses for comparing the convergence rate of each scheme.
Here we compare the performance of the proposed method  $(\ref{eqn:213})$ ($OM8$)  with  the methods (2.13)  ($DF_{8,3}$),  (2.15)  ($DF_{8,4}$) of \cite{Soleymani7}; (4.17) ($ DF_{8,1}$), (4.19)  ($DF_{8,2}$) of \cite{Soleymani2} and (33)  ($DF_{8,5}$) and (35)   ($DF_{8,6}$) of \cite{Soleymani8} respectively.
 The results of comparison of the test functions are summarized in Table 2. From Table 2,  we observe that the new scheme is superior than some existing methods. 
\begin{table}[htbp]
\caption{Functions and their roots.}
\tiny
  \begin{tabular}{|lll|} \hline
$Nonlinear function\ f(x)$                                                        &$\alpha$             & \\ \hline 
$f_1(x)=10xe^{-x^2}-1$                      &$\alpha_1=1.6796...$  &    \\         
 

$f_2(x)=x^2e^x-Sinx$                      &$\alpha_2=0$              &  \\

$f_3(x)=Sin3x+xCosx$ & $\alpha_3=1.19$       &\\ 

$f_4(x)=Log(x)-x^3+2Sin(x)$ & $\alpha_4=1.2979$       &\\ 
$f_5(x)=Cos(x)+Sin(2x) \sqrt{1-x^2}+Sin(x^2)+x^{14}+x^3+\frac{1}{2x}$ & $\alpha_5=-0.92577$       &\\ 
$f_{6}(x)=e^{-x}+Sin(x)-1$ & $\alpha_{6}=2.07683$       &\\ 
$f_{7}(x)=(1+x^3)Cos(\frac{\pi}{2})+\sqrt{1-x^2}-\frac{2(9 \sqrt{2}+7\sqrt{3})}{27}$ & $\alpha_{7}=\frac{1}{3}$       &\\ 
\hline
 \end{tabular}
  \label{tab:abbr}
\end{table}

\begin{table}[!htbp]
\tiny
 \caption{ Comparison of different derivative-free methods.}
\begin{tabular}{|c |c| c|c|c | c|  c|  c |c |cc  |}\hline
$\left|f\right|$ & Guess &  &$DF_{8,1}$&      $DF_{8,2}$&           $DF_{8,3}$&        $ DF_{8,4}$&       $DF_{8,5}$&     $DF_{8,6}$&  $OM8$&\\ \hline
                         &           & IT&   2 &                      2 &                   2 &                  2 &              2 &                2 &        2 &  \\
                         &           & TNE&  8&                      8 &                  8 &                  8 &              8  &                8 &       8  &  \\
$\left|f_1\right|$ & 1.72&  &0.7e-53&      0.3e-53      &0.5e-58&        0.8e-63&      0.5e-61& 0.5e-61& 0.4e-80&\\ \hline
   &           & IT&          3 &         3              &           3             &        3                &        3      &        3         &   3     &  \\
                         &           & TNE&   12   &       12                &       12                 &     12   & 12    &    12  & 12        &  \\
$\left|f_1\right|$ & 1.5&  &0.4e-252&       0.1e-243          &0.2e-238&          0.3e-269&      0.3e-329& 0.3e-329& 0.6e-315&\\ \hline
                          &           & IT&   2 &                      2 &                   2 &                  2 &              2 &                2 &        2 &  \\
                         &           & TNE&  8&                      8 &                  8 &                  8 &              8  &                8 &       8  &  \\
$\left|f_1\right|$ & 1.7&  &0.8e-79&      0.5e-79 &                  0.9e-78&    0.3e-82&    0.9e-82&         0.9e-82&                    0.1e-100&\\ \hline

                         &           & IT&   3        &  3           &     3                   &   3          &         3          &    3               &   2     &  \\
                         &           & TNE& 12    &  12         &   12                     &   12     &   12                & 12                &   8      &  \\
$\left|f_2\right|$ &0.1&  &0.2e-107&  0.1e-114&          0.1e-214&          0.1e-238&    0.1e-231&  0.2e-231& 0.1e-52&\\ \hline
                         &           & IT&  3         &          3             &      3          &        3                &           2        &           2              &2         &  \\
                         &           & TNE&  12         &    12             &  12        &            12            &           8               &     8            &        8 &  \\
$\left|f_2\right|$ &-0.1&  & 0.4e-369&    0.1e-363&0.2e-362&        0.2e-384&             0.1e-52&  0.1e-52&    0.1e-74&\\ \hline
                         &           & IT&      3     &        3     &   3   &           3             &        3           &   3&  3       &  \\
                         &           & TNE&  12         &    12  &   12    &       12           &   12                &   12   &        12 &  \\
$\left|f_2\right|$ &-0.5&  & 0.1e-178&         0.1e-167&            0.5e-162&         0.8e-189&   0.2e-182&   0.2e-182&  0.4e-259&\\ \hline

                         &             & IT&    -   &     -&        - &                  -&        - &   -             & 2        &  \\
                         &           & TNE&  - &             -          &         -               &           -             &   -                &     -            &   8      &  \\
$\left|f_3\right|$ &1.0&      &     DIV.&             NC&                 NC&                 DIV.&             NC&         DIV.& 0.1e-58&\\ \hline

                         &             & IT&     4        &      4              &    -                    &    -                    &    4          &   4            & 3        &  \\
                         &           & TNE&  16         &   16          &   -                     & -                       &       16         &  16          & 12        &  \\
$\left|f_3\right|$ &0.8&      &   0.2e-139&   0.7e-194&                NC&                   NC&  0.3e-310&    0.4e-315&  0.1e-64&\\ \hline

                           &             & IT&       -    &         -    &        3                &       3                 &    3             &   3              & 3        &  \\
                           &           & TNE&      -     &        -   &     12                 &      12                 &    12           & 12                & 12        &  \\
$\left|f_3\right|$ &1.8&      &            DIV.&               DIV.&          0.5e-85&        0.2e-83&    0.1e-75&    0.3e-75&  0.1e-107&\\ \hline

     &             & IT&    3       & 3                      & 3           & 3              &   3         &3          & 3        &  \\
    &           & TNE& 12          &   12       &  12             & 12       & 12      & 12                &  12      &  \\
$\left|f_4\right|$ &1.4&      & 0.4e-214&     0.9e-211&         0.8e-237&           0.1e-253&       0.4e-334&      0.4e-334&  0.4e-333&\\ \hline

     &             & IT& -          &        -             &      -                 &       -                 &      -       &    -             &    3     &  \\
    &           & TNE&   -        &       -                &       -                 &     -        &    -               &      -   &    12     &  \\
$\left|f_4\right|$ &1.15&      &     I&                   I&                NC&                   NC&   NC&    NC&  0.2e-284&\\ \hline

     &             & IT&  2         &         2              &    2                   &      2                      &    2               & 2              & 2        &  \\
     &           & TNE& 8         &         8                    & 8                &     8                           & 8                & 8               & 8         &  \\
$\left|f_4\right|$ &1.3&      &0.3e-153&    0.3e-153&      0.1e-129&    0.5e-133&  0.2e-136&   0.2e-136&  0.1e-161&\\ \hline

    &             & IT&   3        &        3               &       2                 &   2                     &     2              &     2            &   2      &  \\
    &           & TNE&  12     &          12              &  8                      &  8                      &    8               & 8                & 8        &  \\
$\left|f_5\right|$ &-0.92&  &0.4e-295&     0.2e-301&    0.4e-57&      0.2e-74&    0.6e-59&    0.6e-59& 0.8e-97&\\ \hline

    &             & IT& 2          &        2               &       2                 &      2                  &   2                &     2            &     2    &  \\
    &           & TNE&   8        &        8               &         8               &       8                 &     8              &     8            &   8      &  \\
$\left|f_5\right|$ &-0.93& & 0.3e-61&  0.6e-61&     0.1e-65&  0.5e-70&  0.1e-70&  0.1e-70&  0.4e-98&\\ \hline

    &             & IT&  -         &        -               &             3           &      3                  &     3              &   3              &     3    &  \\
    &           & TNE& -        &     -                  &           12             &   12                  &  12               & 12               &12         &  \\
$\left|f_5\right|$ &-0.9&  &    I&     I&    0.1e-108&    0.2e-123& 0.8e-104&   0.1e-103& 0.2e-361&\\ \hline

    &             & IT&          3 &                     3  &                3        &    3                    &   3                &  3               &  2     &  \\
    &           & TNE&      12 &               12       &         12              &   12                   &12               &12            &8        &  \\
$\left|f_6\right|$ &1.9&  &0.4e-231&   0.6e-236&     0.5e-298&    0.2e-328&    0.1e-301&     0.4e-334&  0.3e-59&\\ \hline

    &             & IT&    3       &       3                &    3                    & 3                       &  2                 &     3            &    2     &  \\
   &           & TNE&   12     &     12                  & 12                 &12                        & 8                   &12                 &  8      &  \\
$\left|f_6\right|$ &2.3&  &   0.2e-354&   0.2e-351&      0.3e-336&     0.1e-357&   0.6e-54&   0.2e-301&  0.4e-60&\\ \hline

    &             & IT&    4       &      4      &   3                     &    3                    &     3              &    3             &    3     &  \\
    &           & TNE& 16      &    16       &   12                     &  12                      &   12                & 12      & 12       &  \\
$\left|f_6\right|$ &1.8&  & 0.1e-200& 0.1e-244&   0.8e-172&    0.6e-192& 0.9e-153&  0.2e-152& 0.7e-352&\\ \hline

   &             & IT&  3         &     3                  &     3                   &     3                   &      3             &  3               &  3       &  \\
   &           & TNE&12        &   12                    & 12                        &  12                      &    12               &12               & 12         &  \\
$\left|f_7\right|$ &0.8&  & 0.2e-57&   0.1e-55&    0.2e-72&  0.2e-82&   0.3e-98&   0.7e-202&  0.8e-219&\\ \hline

   &             & IT&  3             &   3                    &  3                     &     3                   &  3                &  3               &   3      &  \\
   &           & TNE& 12          &   12        & 12                       & 12                            & 12                 &  12               &   12      &  \\
$\left|f_7\right|$ &0.6&  & 0.1e-161&  0.2e-162&    0.1e-159&    0.1e-172&   0.6e-202&  0.2e-301&  0.2e-378&\\ \hline

   &             & IT&  2         &       2                &       2                 &       2                 &   2                &      2           &  2       &  \\
   &           & TNE&  8         &    8                   &     8                   &      8                  &  8                 &     8            & 8        &  \\
$\left|f_7\right|$ &0.4&  &0.5e-60&  0.6e-60& 0.1e-52&   0.7e-55& 0.2e-62& 0.2e-62& 0.1e-70&\\ \hline

 \end{tabular}
 \label{tab:abbr}
\end{table}

\newpage
\section{Concluding remarks}
In the present study, we have contributed a class  eighth-order Steffensen-type method. We have also described the dynamical behavior of some eight-order derivative-free methods. The proposed method is free from derivative. Our new scheme has maximum possible efficiency  index using four evaluations. Some numerical examples have been carried out to confirm the underlying theory of this study. From numerical  results one can observe that our contributed scheme is superior than some existing methods.

\textsc{Anuradha Singh\\
Department of Mathematics, \\
Maulana Azad National Institute of Technology,\\
Bhopal, M.P., India-462051}.\\
E-mail: {singhanuradha87@gmail.com; singh.anu3366@gmail.com}.\\\\
\textsc{Jai Prakash Jaiswal\\
Department of Mathematics, \\
Maulana Azad National Institute of Technology,\\
Bhopal, M.P., India-462051}.\\
E-mail: {asstprofjpmanit@gmail.com; jaiprakashjaiswal@manit.ac.in}.

\begin{thebibliography}{10}
\bibitem{Kung}
H. T. Kung and J. F. Traub: Optimal order of one-point and multipoint iteration, JCAM, 21 (1974), pp. 643-651 .
\bibitem{steffensen}
J. F. Steffensen : Remarks on iteration, skand Aktuar Tidsr, 16 (1933), pp. 64-72 .
\bibitem{Traub1} 
J. F. Traub : Iterative methods for solution of equations, Chelsea Publishing, New York, USA (1997).
\bibitem{Gautschi}
W. Gautschi : Numerical Analysis-An Introduction, Birkhauser, Barton Mass, USA (1997).
 \bibitem{Soleymani1}
F. Soleymani and B. S. Mousavi : On Novel Classes of Iterative Methods for Solving Nonlinear Equations, Computational Mathematics and Mathematical Physics, 2012, Vol. 52, No. 2, pp. 203-210.
 \bibitem{Sharma}
J. R. Sharma and H. Arora: An efficient family of weighted-Newton methods with optimal eight order convergence, Appl. Math. Lett., 2014, Vol. 29, pp.1-6.
 \bibitem{Babajee}
D. K. R. Babajee, A. Cordero, F. Soleymani and J. R. Torregrosa : On improved three-step schemes with high efficiency index and their dynamics, Numer. Algor., 2014, Vol. 65, pp.153-169.
 \bibitem{Chun}
C. Chun and M. Y. Lee : A new optimal eight-order family of iterative methods for the solutio of nonlinear equations, Appl. Math. and Comp., 2013, Vol. 223, pp. 506-519.
 \bibitem{Soleymani}
F. Soleymani : New class of eight-order iterative zero-finders and their basins of attraction, Afr. Mat., 2014, Vol.25, pp.67-79.
 \bibitem{Soleymani11}
A. Cordero, T. Lotfi, K. Mahdiani and J. R. Torregrosa : Two optimal general classes of iterative methods with eight-order, Acta. Appl Math, 2014, DOI 10.1007/s10440-014-9869-0. 


\bibitem{Soleymani3}
F. Soleimani, F. Soleymani and S. Shateyi : Some iterative methods free from derivatives and their basins of attraction for nonlinear equations, Discrete Dynamics in Nature and Society ,2013, Article ID 301718, 10 pages.

\bibitem{carlos}
C. Andreu, N. Cambil, A. Cordero and J. R. Torregrosa : A class of optimal eighth-order derivative free methods for solving the Danchick-Gauss problem, Appl. Math. Comput., 232 (2014), pp. 237-246.


\bibitem{Soleymani2}
F. Soleymani : Optimized Steffensen-type methods with eighth-order convergence and high efficiency index, International Journal of Mathematics and Mathematical Sciences , 2012, Article ID 932420, 18 pages.
\bibitem{santiago}
S. Artidiello, A. Cordero, J. R. Torregrosa and M. P. Vassileva : Two weighted eight-order classes of iterative root-finding methods, Int. Jour. of Comp. Mathe. , 2014, DOI 10.1007/s10440-014-9869.

\bibitem{Soleymani4}
F. soleymani : Efficient optimal eighth-order derivative free methods for non linear equations, Japan J. Indust. Appl. Math., 30 (2013), pp. 287-306.

\bibitem{Soleymani5}
F. Soleymani , D. K. R Babajee, S. Shateyi, S. S. Motsa : Construction of optimal derivative-free techniques without memory, J. Appl. Math., 2012, Article ID 497023, 24 pages.

\bibitem{Soleymani6}
F. Soleymani , D. K. R Babajee, S. Shateyi, S. S. Motsa : Optimal Steffensen-type methods with eighth-order of convergence, Comput. Math. Appl., 62 (2011), 4619-4626.

\bibitem{Soleymani7}
F. Soleymani and S. Shateyi : Two optimal eighth-order derivative free classes of iterative methods, Abst. Appl. Anal., 2012,  Article ID 318165, 14 pages.
\bibitem{Soleymani8}
F. Soleymani : Some optimal iterative methods and their with memory variants, Journal of Egyptian Mathematical Society, 21 (2013), pp. 133-141.\\




\end{thebibliography}
\end{document}